\documentclass[a4paper]{amsart}

\usepackage{amssymb}
\usepackage[leqno]{amsmath}
\usepackage{hyperref}

\DeclareMathOperator{\add}{add}

\newcommand{\Ww}{\mathcal{W}}

\newcommand\ZZ{\mathbb{Z}}

\newcommand\XX{\mathbb{X}}
\newcommand\YY{\mathbb{Y}}

\newcommand\QQ{\mathbb{Q}}
\newcommand{\WW}{\mathbb{W}}
\newcommand\oQQ{\overline{\mathbb{Q}}}

\newcommand\rad{{\mathrm{rad}\,}}

\newcommand\rank{\operatorname{rk}}

\newcommand\REM[1]{}
\newcommand\Der{\operatorname{D^b}}
\newcommand\md{\operatorname{mod}}
\newcommand\coh{\operatorname{coh}}
\newcommand\ind{{\mathrm{ind}}}
\newcommand\Her{\mathcal{H}}
\newcommand\tHer{{\widetilde{\mathcal{H}}}}
\newcommand\Clu{\mathcal{C}}
\newcommand\dual{\operatorname{D}}
\newcommand\Ext{\operatorname{Ext}}
\newcommand\Aut{\operatorname{Aut}}
\newcommand\Pic{\operatorname{Pic}}
\newcommand\End{\operatorname{End}}
\newcommand\Hom{\operatorname{Hom}}

\newcommand\SL{\operatorname{SL}}
\newcommand\PSL{\operatorname{PSL}}
\newcommand\Groth{\operatorname{K_0}}
\newcommand\spitz[1]{\langle #1\rangle}

\newcommand{\Ii}{\mathcal{I}}

\newcommand{\ra}{\rightarrow}
\newcommand{\xra}[1]{\xrightarrow{#1}}
\newcommand{\Tt}{\mathcal{T}}
\newcommand{\Ff}{\mathcal{F}}
\newcommand{\boldlambda}{\text{\boldmath$\lambda$}}
\newcommand{\boldp}{\text{\boldmath$p$}}
\newcommand{\kel}{{\mathrm{kel}}}

\newcommand{\oT}{\overline{T}}

\newcommand{\graph}[1]{\Gamma_{#1}}
\newcommand{\comp}[1]{\langle #1\rangle}
\newcommand{\si}{\sigma}
\newcommand{\Oo}{\mathcal{O}}
\newcommand{\vx}{\vec{x}}
\newcommand{\vc}{\vec{c}}

\newcommand{\Ga}{\Gamma}

\newtheorem{Lemma}{Lemma}[section]
\newtheorem{Proposition}[Lemma]{Proposition}
\newtheorem{Theorem}[Lemma]{Theorem}
\newtheorem{Corollary}[Lemma]{Corollary}
\theoremstyle{definition}

\newtheorem{Remark}[Lemma]{Remark}

\newcommand{\proofend}{\hfill$\Box$\par}
\newcommand{\HVCenter}[1]{\setbox 0=\hbox{#1}%
        \dimen0=\wd0%
        \dimen1=\ht0%
        \divide\dimen0 by 2%
        \divide\dimen1 by 2%
        \hskip -\dimen0%
        \lower \dimen1%
        \box0%
        \hskip -\dimen0}
\newcommand{\HBCenter}[1]{\setbox 0=\hbox{#1}%
        \dimen0=\wd0%
        \dimen1=\ht0%
        \divide\dimen0 by 2%
        \hskip -\dimen0%
        \box0%
        \hskip -\dimen0}
\newcommand{\HTCenter}[1]{\setbox 0=\hbox{#1}%
        \dimen0=\wd0%
        \dimen1=\ht0%
        \divide\dimen0 by 2%
        \hskip -\dimen0%
        \lower \dimen1%
        \box0%
        \hskip -\dimen0}
\newcommand{\RVCenter}[1]{\setbox 0=\hbox{#1}%
        \dimen0=\wd0%
        \dimen1=\ht0%
        \divide\dimen1 by 2%
        \hskip -\dimen0%
        \lower \dimen1%
        \box0%
        \hskip -\dimen0}
\newcommand{\LVCenter}[1]{\setbox 0=\hbox{#1}%
        \dimen1=\ht0%
        \divide\dimen1 by 2%
        \lower \dimen1%
        \box0%
        \hskip -\dimen0}

\begin{document}

\title{The cluster category of a canonical algebra}
\author{M.~Barot}
\address{Instituto de Matem\'aticas\\
  Universidad Nacional Aut\'onoma de M\'exico\\
  Ciudad Universitaria, C.P.\ 04510\\
  Mexico}
\email{barot@matem.unam.mx}
\author{D.~Kussin}
\address{Institut f\"ur Mathematik\\
Universit\"at Paderborn\\
33095 Paderborn\\
Germany}
\email{dirk@math.uni-paderborn.de}
\author{H.~Lenzing}
\address{Institut f\"ur Mathematik\\
Universit\"at Paderborn\\
33095 Paderborn\\
Germany}
\email{helmut@math.uni-paderborn.de}


\begin{abstract}
  We study the cluster category of a canonical algebra $A$ in terms of
  the hereditary category of coherent sheaves over the corresponding
  weighted projective line $\XX$. As an application we determine the
  automorphism group of the cluster category and show that the
  cluster-tilting objects form a cluster structure in the sense of
  Buan-Iyama-Reiten-Scott. The tilting graph of the sheaf category
  always coincides with the tilting or exchange graph of the cluster
  category. We show that this graph is connected if the Euler
  characteristic of $\XX$ is non-negative, or equivalently, if $A$ is
  of tame (domestic or tubular) representation type.
\end{abstract}
\maketitle
\sloppy

\section{Introduction}
Cluster categories $\Clu(H)$ of a hereditary algebra $H$ were
introduced in \cite{5clu} as orbit categories $\Der(H)/F^\ZZ$ of the
derived category of $H$, where $F=\tau^-\circ [1]$. These categories
have been extensively studied due to their remarkable connections to
cluster algebras in the sense of Fomin-Zelevinsky~\cite{FZ1}. In this
paper we study the cluster category $\Clu(A)$ of a canonical algebra
$A$.  It was shown in~\cite{Ke} that also in this case $\Clu(A)$ is a
triangulated category. For information on canonical algebras we refer
to~\cite{GeLe,Ri}.

We give an important alternative description of the cluster category
$\Clu=\Clu(A)$ in terms of the hereditary category $\Her=\coh(\XX)$ of
coherent sheaves on the weighted projective line $\XX$ attached to
$A$, namely $\Clu$ can be obtained from $\Her$ by ``adding extra
morphisms''.  Conversely, $\Her$ can be recovered from $\Clu$ as the
quotient $\Clu/\Ii$ by a suitable ideal, see
Corollary~\ref{cor:recover}. Note that our setting includes the
cluster categories of tame hereditary algebras.

The cluster categories studied here provide a particular family of
$2$-Calabi-Yau triangulated categories~\cite{Ke} having a cluster
structure in the sense of \cite{BIRS}, see Theorem~\ref{thm:tub-ca}.
These cluster categories give rise to further cluster-tilted algebras.
Note that cluster categories associated to hereditary abelian
categories that are not module categories have also been considered
in~\cite{5clu} and~\cite{Zhu}.

Another main result is a lifting property which states that exact
autoequivalences of the cluster category can be lifted to exact
autoequivalences of $\Der(\Her)$, see Theorem~\ref{thm:lifting}. For
this a detailed study of cluster tubes is crucial. We show that the
automorphism group $\Aut(\Clu)$ is canonically isomorphic to
$\Aut(\Der(\Her))$ modulo the cyclic subgroup generated by $F$, see
Theorem~\ref{Thm:Aut1}.  If $A$ is not tubular, then $\Aut(\Clu)$ is
canonically isomorphic to $\Aut(\Her)$.  In the tubular case we
establish a natural bijection between the coset space
$\Aut(\Clu)/\Aut(\Her)$ and $\QQ\cup\{\infty\}$, see
Theorem~\ref{thm:AutC-AutH}.

As another application the investigation of the automorphism group in
the tubular case yields -- combined with results of
H\"{u}bner~\cite{Huebner} -- the connectedness of the tilting graph in
the tubular case.

\section{Basic Results}
\label{sec:Basics}

\noindent
Throughout this article, $k$ will denote an algebraically closed
field.

\subsection*{Definition of the cluster category}

Let $A$ be a canonical algebra of weight type $\boldp=(p_1,\ldots,p_t)$ and
parameter sequence $\boldlambda=(\lambda_3,\ldots,\lambda_t)$, see
\cite{Ri,GeLe}. We denote by $\md(A)$ the category of finitely
generated (right) $A$-modules and by $\Der(A)$ the bounded derived
category of $\md(A)$. Recall that the Auslander-Reiten translation
$\tau$, its inverse $\tau^-$ and the shift $[1]$ are autoequivalences of
$\Der(A)$.

By~\cite{GeLe}, $\Der(A)$ is equivalent to $\Der(\Her)$, where
$\Her=\coh(\XX)$ is the category of coherent sheaves on a weighted
projective line $\XX$ of type $(\boldp,\boldlambda)$.  Note that
$\Her$ is a connected abelian hereditary category with Serre duality
$\dual \Ext^1_{\Her}(X,Y)=\Hom_\Her(Y,\tau X)$ for all $X,\,Y\in\Her$,
where $\tau$ is an autoequivalence of $\Her$.  Further $\Her$ has a
tilting object $T$ such that $\End_\Her (T)=A$.

We study the \emph{cluster category} $\Clu=\Clu(\Her)$, defined as the
\emph{orbit category} $\Der(\Her)/F^\ZZ$, where $F=\tau^-\circ [1]$,
see \cite{5clu}; it has the same objects as $\Der(\Her)$, morphism
spaces are given by $\bigoplus_{n\in\ZZ}\Hom_{\Der(\Her)}(X,F^nY)$
with the obvious composition. We denote by $\pi\colon
\Der(\Her)\rightarrow\Clu$ the natural projection functor and call a
triangulated structure on $\Clu$ \emph{admissible} if $\pi$ becomes an
exact functor; the shift in $\Clu$ is then given by $\tau$.
By~\cite{Ke} such an admissible triangulated structure always exists.

\subsection*{Description of $\Clu$ in terms of $\Her$}

Let $\tHer$ be the category with the same objects as $\Her$ and
$\ZZ_2$-graded morphism spaces
$\Hom_{\tHer}(X,Y)=\Hom_\Her(X,Y)\oplus\Ext_\Her^1(X,\tau^- Y)$. Then
each morphism $f\in\Hom_{\tHer}(X,Y)$ can be written as $f=f_0+f_1$
for some $f_0\in\Hom_\Her(X,Y)$ of degree zero and some
$f_1\in\Ext_\Her^1(X,\tau^- Y)$ of degree one. The composition is
given by $(g_0+g_1)\circ (f_0+f_1)=g_0 f_0+(\tau^{-} g_0 f_1+g_1f_0)$
using the Yoneda product.  Since $\tau^-$ is an autoequivalence of
$\Her$ we obtain the following result.

\begin{Proposition}\label{prop:comp_in_tHer}
  The category $\Clu$ is equivalent to the category $\tHer$.
  \proofend
\end{Proposition}

In the sequel we shall identify $\Clu$ with $\tHer$, that is, we view
$\Clu$ as the category $\Her$ with extra morphisms of degree one.
Fixing an admissible triangulated structure, $\Clu$ becomes a
$2$-Calabi-Yau triangulated category, see \cite{Ke}.  We call a
triangle in $\Clu$ \emph{induced} if it is isomorphic to the image of
an exact triangle in $\Der(\Her)$ under the projection functor.

\begin{Remark}\label{Remark:factorIdeal}
  In $\Clu$ the composition of two morphisms of degree one is zero,
  hence the degree one morphisms form an ideal $\Ii$ contained in the
  radical $\rad_\Clu$ of $\Clu$. Clearly $\Her\simeq \Clu/\Ii$, that
  is, $\Her$ can be recovered from $\Clu$ provided that the above
  $\ZZ_2$-grading is known. We shall later see how the $\ZZ_2$-grading
  can be obtained intrinsically in terms of the category $\Clu$, see
  Corollary~\ref{cor:recover}.
\end{Remark}

\noindent
For the notion of a \emph{cluster-tilting object} in $\Clu$ we refer
to \cite{5clu}.

\begin{Proposition}\label{prop:CandH}
  The cluster category $\Clu$ is a Krull-Remak-Schmidt category. The
  categories $\Clu$ and $\Her$ have the same indecomposables, the same
  isomorphism classes of objects and the tilting objects in $\Her$ are
  precisely the cluster-tilting objects in $\Clu$.
\end{Proposition}

\begin{proof}
  We first observe that a morphism $f=f_0+f_1$ is invertible in $\Clu$
  if and only if $f_0$ is invertible in $\Her$ since $f_1$ is radical.
  Therefore isomorphism classes coincide in both categories.  Since
  $\End_\Clu(X)/\rad_\Clu(X,X)=\End_\Her(X)/\rad_\Her(X,X)$, the
  categories $\Clu$ and $\Her$ have the same indecomposables and
  $\Clu$ is a Krull-Remak-Schmidt category.  The last assertion
  follows from the definitions.
\end{proof}

\noindent
The following result can be proved as in \cite[Prop.~1.3]{5clu}.

\begin{Proposition}\label{prop:AR-structure}
  The cluster category $\Clu$ has Auslander-Reiten triangles which
  coincide with the triangles induced by almost-split sequences in
  $\Her$. Moreover, $\Her$ and $\Clu$ have the same Auslander-Reiten
  quiver.\proofend
\end{Proposition}

\begin{Corollary}
  For objects $X,Y$ of\,\ $\Clu$ the space $\Hom_\Clu(X,Y)_1$ of
  degree one morphisms is contained in the infinite radical
  $\rad_\Clu^\infty(X,Y)$.
\end{Corollary}

\begin{proof}
  Since $X$ and $FY$ belong to distinct Auslander-Reiten components in
  $\Clu$ we have $\Hom_\Clu(X,Y)_1=\Hom_{\Der(\Her)}(X,FY)\subseteq
  \rad^\infty_{\Der(\Her)}(X,FY)$, hence the result.
\end{proof}

\begin{Remark}
  There are further hereditary categories allowing a treatment by the
  techniques of this paper. If $H$ is a connected hereditary algebra
  of infinite representation type, then $\Der(H)=\Der(\Her)$ for some
  hereditary category $\Her$ for which $\tau$ is an autoequivalence
  (take $\Her=\mathcal{I}[-1]\vee \mathcal{P}\vee \mathcal{R}$, where
  $\mathcal{P}$, $\mathcal{I}$ are the preprojective, respectively
  preinjective component and $\mathcal{R}$ consists of all regular
  components of $\md(H)$). The tame hereditary case is covered by our
  setting.
\end{Remark}

\subsection*{Shape of the cluster category}

Denote by $\Her_0$ (resp.\ $\Her_+$) the full subcategory of $\Her$
consisting of the objects of finite length (resp.\ the vector
bundles). We define $\Clu_0$ (resp.\ $C_+$) as the full subcategory of
$\Clu$ given by the objects of $\Her_0$ (resp.\ of $\Her_+$). A full
subcategory of $\Clu$ given by the objects of a tube in $\Her$ will be
called \emph{cluster tube}.

The \emph{slope} function $\mu$ assigns to each non-zero object $X$ of
$\Clu$ an element of $\oQQ=\QQ\cup\{\infty\}$ by
$\mu(X)=\deg(X)/\rank(X)$, where $\deg$ and $\rank$ are the degree and
the rank functions on $\Her$, respectively, see~\cite{GeLe}. Recall
that the \emph{Euler characteristic} $\chi_\Her=2-\sum_{i=1}^t
(1-1/p_i)$ determines the representation type: if $\chi_\Her> 0$
(resp.\ $=0,<0$) then $\Her$ is tame domestic (resp.\ tubular, wild),
see~\cite{GeLe}.

In the tubular case for each $q\in\oQQ$, we denote by $\Clu^{(q)}$ the
full subcategory of $\Clu$ given by the additive closure of all
indecomposables of slope $q$. In particular we have
$\Clu^{(\infty)}=\Clu_0$.  For $q<q'$ (resp. $q>q'$) each non-zero
morphism from $\Clu^{(q)}$ to $\Clu^{(q')}$ is of degree zero (resp.
one).

By definition a \emph{tubular family} in $\Clu$ is a maximal family of
pairwise orthogonal cluster tubes. If $\chi_\Her\neq 0$ then $\Clu_0$
is the unique tubular family in $\Clu$.  If $\Her$ is tubular then the
categories $\Clu^{(q)}$ (for $q\in\oQQ$) are precisely the tubular
families, see~\cite{LeMe2}.

\section{Relationship to Fomin-Zelevinsky mutations}

\subsection*{Cluster structure}

By Proposition~\ref{prop:CandH} the cluster-tilting objects in $\Clu$
correspond to the tilting objects in $\Her$.  For a cluster-tilting
object $T$ we denote by $Q_T$ the quiver of the endomorphism algebra
$\End_\Clu(T)$. We call an object $E\in\Her$ \emph{exceptional\/}
  if it is indecomposable with $\Ext^1_\Her (E,E)=0$.

We know from \cite{Huebner} that in $\Her$ for each indecomposable
direct summand $M$ of $T$, that is, $T=\oT\oplus M$ there exists a
unique exceptional object $M^\ast\not\simeq M$ in $\Her$ such that
$T^\ast=\oT\oplus M^\ast$ is again a tilting object.

A $2$-Calabi-Yau triangulated category $\Clu$ with finite dimensional
Hom-spaces \emph{admits a cluster structure}~\cite{BIRS} if (i) for
each cluster-tilting object $T$ the quiver $Q_T$ has no loop and no
$2$-cycle, and (ii) if $T=\overline{T}\oplus M$ and
$T^\ast=\overline{T}\oplus M^\ast$ are two cluster-tilted objects with
non-isomorphic indecomposables $M$ and $M^\ast$, then $Q_{T^\ast}$ is
the Fomin-Zelevinsky mutation \cite{FZ1} of $Q_T$ in the vertex
corresponding to $M$.

\begin{Theorem}
  \label{thm:tub-ca}
  For any canonical algebra the category $\Clu$ admits a cluster
  structure.
\end{Theorem}

\begin{proof}
  By \cite[Thm.~I.1.6]{BIRS} we only have to show condition~(i).  We
  know from \cite{IyYo} that in $\Clu$ there exist exact triangles
  $M^\ast\xra{u} B\xra{v} M\ra M^\ast[1]$ and $M\xra{u'} B'\xra{v'}
  M^\ast\ra M[1]$, where $u$, $u'$ are minimal left and $v$, $v'$ are
  minimal right $\add(\oT)$-approximations. By \cite{Huebner} one of
  them, say the first, is induced by an exact sequence in $\Her$. As
  in \cite[Lem.~6.13]{5clu} we apply $\Hom_{\Der(\Her)}(F^{-1}M,-)$ to
  the first sequence and get that each radical morphism $f\colon M\ra
  M$ factors through the morphism $v$.  Hence $f$ is not irreducible
  in $\End_\Clu(T)$. This shows that there are no loops in $Q_T$.

  To see that there are no $2$-cycles in $Q_T$ it is enough to show
  that $B$ and $B'$ have no common indecomposable summand.  Assume
  that such a summand $U$ exists. Since the first triangle is induced
  by an exact sequence from $\Her$, the morphisms $u$, $v$ in $\Clu$
  are of degree zero. Denote by $j\in\Hom_\Her (U,B')$ and by
  $p\in\Hom_\Her (B',U)$ the canonical inclusion and projection,
  respectively. Since the endomorphism algebra $\End_\Her(T)$ is
  triangular~\cite[Lem.~IV.1.10]{Hap}, the non-zero morphisms
  $s=p\circ u'$ and $t=v' \circ j$ in $\Clu$ are of degree one. Thus
  $s$ is given by a non-split short exact sequence. This sequence
  induces a triangle $\tau^-U\xra{f} E\xra{g} M\xra{s}\tau^- U[1]$ in
  $\Der(\Her)$, yielding by rotation an induced triangle $M\xra{s}
  U\xra{\tau f} \tau E\xra{\tau g} \tau M$ in $\Clu$. Since $s=p\circ
  u'$, by axiom~\cite[(TR3)]{Ve} there exists a morphism
  $q\in\Hom_\Clu (M^\ast,\tau E)$ such that $q\circ v' =\tau f \circ
  p$. Therefore the degree zero morphism $\tau f=\tau f\circ p\circ j$
  equals the degree one morphism $q\circ t$, showing that both are
  zero. This yields a contradiction, since $f\neq 0$.
\end{proof}

\subsection*{Cluster-tilted algebras}

Considering the endomorphism rings of cluster-tilting objects we get a
new class of algebras.  For example, if $A$ is a canonical algebra
with $t=3$ weights and if $T$ is a tilting object in $\Her$ for which
$\End_\Her(T)=A$ then $A_\Clu=\End_\Clu(T)$ is given as quotient of
the path algebra of the following quiver $Q_T$ modulo the ideal $I$
generated by the elements described below (in $Q_T$ the arm with
arrows $x_i$ contains $p_i$ arrows).

\begin{center}
  \begin{picture}(300,58)
    \put(0,4){
    \put(-5,47){\RVCenter{$Q_T:$}}
    \put(0,22){
      \multiput(0,0)(143,0){2}{\circle*{3}}
      \multiput(0,-20)(0,40){2}{%
        \multiput(30,0)(20,0){2}{\circle*{3}}
        \multiput(93,0)(20,0){2}{\circle*{3}}
        \multiput(34,0)(63,0){2}{\vector(1,0){12}}
        \put(54,0){\line(1,0){8}}
        \multiput(66,0)(5,0){3}{\line(1,0){1}}
        \put(81,0){\vector(1,0){8}}
      }
      \put(0,10){
        \multiput(30,0)(20,0){2}{\circle*{3}}
        \multiput(93,0)(20,0){2}{\circle*{3}}
        \multiput(34,0)(63,0){2}{\vector(1,0){12}}
        \put(54,0){\line(1,0){8}}
        \multiput(66,0)(5,0){3}{\line(1,0){1}}
        \put(81,0){\vector(1,0){8}}
      }
      \multiput(3,2)(113,-20){2}{\vector(3,2){24}}
      \multiput(3,-2)(113,20){2}{\vector(3,-2){24}}
      \put(3.6,1.2){\vector(3,1){22.8}}
      \put(116.6,8.8){\vector(3,-1){22.8}}
      \put(10,11){\HBCenter{\small $x_1$}}
      \put(10,-11){\HTCenter{\small $x_3$}}
      \put(20,0){\HBCenter{\small $x_2$}}
      \put(133,11){\HBCenter{\small $x_1$}}
      \put(133,-11){\HTCenter{\small $x_3$}}
      \put(123,0){\HBCenter{\small $x_2$}}
      \multiput(40,23)(63,0){2}{\HBCenter{\small $x_1$}}
      \multiput(40,7)(63,0){2}{\HTCenter{\small $x_2$}}
      \multiput(40,-17)(63,0){2}{\HBCenter{\small $x_3$}}
      \qbezier(138,-1)(120,-7)(71.5,-7)
      \qbezier(5,-1)(23,-7)(71.5,-7)
      \put(5,-1){\vector(-3,1){0.01}}
      \put(71.5,-4){\HBCenter{\small $\eta$}}
    }
    \put(175,0){
      \put(-5,47){\RVCenter{$I:$}}
      \put(0,35){$x_1^{p_1}+x_2^{p_2}+x_3^{p_3}$}
      \put(0,20){$x_i^{p_i-a}\eta x_i^{a-1}$\quad for
        $a=1,\ldots,p_i$.}
      \put(60,7){and $i=1,\,2,\,3$.}
    }
    }
  \end{picture}
\end{center}
If $t\leq 2$ then $A_\Clu=A$. For $t\geq 4$ the description of
$A_\Clu$ is more complicated since the relations for the canonical
algebra contain parameters.

\section{Factorization}
\label{sec:Factorization}

\subsection*{Degree one morphisms}

We start with a general result.

\begin{Proposition} \label{prop:factor} Let $X$ and $Y$ be
  indecomposables in $\Clu_+$ and let $\Tt$ be a cluster tube
  from $\Clu_0$. Then each morphism $f\colon X\ra Y$ of degree one
  factors through $\Tt$.
\end{Proposition}

\begin{proof}
  By definition, we have $f=\eta$ with
  $\eta\in\Ext^1_\Her(X,\tau^{-}Y)$. In $\Der(\Her)$ the morphism
  $\eta\colon X\ra \tau^{-}Y[1]$ factors through some object $Z$ from
  $\Tt$, that is,
  $\eta=[X\xrightarrow{h}Z\xrightarrow{\eta'}\tau^{-}Y[1]]$.  Note
  that $h\in\Hom_\Clu(X,Z)$, $\eta'\in\Hom_\Clu(Z,Y)$, hence
  $\eta=\eta'\,h$ yields a factorization of $f$.
\end{proof}

\subsection*{The tubular case}

We now turn to the tubular case.  For $p,q\in\oQQ$, we define the
\emph{slope interval} from $p$ to $q$ as
$$
(p,q)_\Clu\colon =\begin{cases}
  (p,q),&\quad\text{ if }p<q,\\
  (p,\infty]\cup (-\infty,q),&\quad\text{ if }p>q,\\
  \oQQ\setminus\{q\},&\quad\text{ if }p=q,
\end{cases}
$$
where the intervals occurring on the right side are taken in
$\oQQ$. For further properties of $\oQQ$ we refer to
Section~\ref{sec:automorphisms}.

The following result states an important property of $\Clu$ which is
analogous to the well-known factorization properties for the derived
category of a tubular algebra.
\begin{Theorem}\label{Thm:Factorization} We assume that $\Clu$ is the
  cluster category of a tubular category $\Her$.
Let $p,q\in \oQQ$ and $f\colon X\rightarrow Y$ be a non-zero
morphism in $\Clu$ with indecomposable objects $X\in\Clu^{(p)}$
and $Y\in\Clu^{(q)}$.
 \begin{itemize}
\item[{\rm (i)}]
Let $p\neq q$.  Then $f$ factors through each tube of $\Clu^{(r)}$ if and
only if $r\in(p,q)_\Clu$.
\item[{\rm (ii)}]
Let $p=q$. Then
$f$ factors through each tube of $\Clu^{(r)}$ for each $r\neq q$ if and
only if $f$ lies in $\rad^\infty_\Clu$.
\end{itemize}
\end{Theorem}

\begin{proof}
  Note first that each morphism $g\in\Hom_\Clu(\Clu^{(s)},\Clu^{(t)})$
  has degree zero (resp.\ one) if $s<t$ (resp.\ $s>t$). Now if $p<q$
  then $\Hom_\Clu(X,Y)=\Hom_\Her(X,Y)$ and by \cite{LeMe2} the
  morphism $f$ factors through each tube of $\Clu^{(r)}$ if
  $r\in(p,q)_\Clu$. For $r=p$ or $r=q$ clearly $f$ does not factor
  through each tube of $\Clu^{(r)}$. In the remaining cases where
  $r\notin (p,q)_\Clu$, each composition $X\xrightarrow{g}Z
  \xrightarrow{h}Y$ with $Z\in\Clu^{(r)}$ has degree one and therefore
  $hg\neq f$.

If $p>q$ then $f\in\Hom_\Clu(X,Y)=\Hom_{\Der(\Her)}(X,\tau^{-}Y)$
factors through each tube of $\Her^{(r)}$ (for $r\in (p,\infty)$) and
each tube of $\Her^{(r)}[1]$ (for $r\in(-\infty,q)$), thus through any
cluster tube of $\Clu^{(r)}$ for $r\in(p,q)_\Clu$. The case $r=p$ or $r=q$ is
similar as before and for the remaining slopes $r\notin(p,q)_\Clu$
each composition $hg$ ($h,g$ and $Z$ as before) is zero as composition
of morphisms of degree one. This proves assertion (i).

For (ii) we observe that the degree one part $f_1$ of $f$ belongs to
$\Hom_{\Der(\Her)}(X,FY)$ hence also to the infinite radical of
$\Der(\Her)$. Thus $f_1$ factors through any tube of $\Her^{(r)}[1]$
for $r\in(p,p)_\Clu$, and the assertion is true if $f_0=0$. If $f_0$
is non-zero then $f_0$ does not belong to the infinite radical of
$\Her$ and therefore $f$ does not belong to the infinite radical of
$\Clu$. This shows that $X$ and $Y$ lie in the same tube. The result
follows.
 \end{proof}

\subsection*{Recovering $\Her$ from $\Clu$}
\label{subsec:recover_H_from_C}

We now focus on the problem of reconstructing the original category
$\Her$ from $\Clu$ (up to equivalence), assuming only intrinsic
properties of the
category $\Clu$.

We know already from Remark~\ref{Remark:factorIdeal} that the
morphisms of degree one form a two-sided ideal $\Ii$ such that
$\Clu/\Ii\xrightarrow{\sim} \Her$ canonically. We are now going to
show how $\Ii$ can be recovered from $\Clu$ intrinsically. This is
possible without any extra choice if $\Her$ has Euler characteristic
different from zero. In the tubular
case  the reconstruction will depend on the choice of a
tubular family $\Clu^{(q)}$ in $\Clu$. By means of a suitable
autoequivalence of $\Der(\Her)$, see~\cite{LeMe2}, we may then choose
$q=\infty$, that is, $\Clu^{(q)}=\Clu_0$.

\begin{Proposition} \label{prop:ideal}
Regardless of the Euler characteristic, the ideal $\Ii$ of degree one
morphisms is given on indecomposables $X$ and $Y$ by 
$$
\Ii(X,Y)=
\begin{cases}
0 & \text{if } X\in \Clu_+ ,\, Y\in\Clu_0\\
\Hom_\Clu(X,Y) & \text{if } X\in\Clu_0,\, Y\in\Clu_+\\
\rad^\infty_\Clu(X,Y) & \text{if } X,Y\in\Clu_0\\
\Ff(X,Y)& \text{if } X,Y\in\Clu_+
\end{cases}
$$
where $\Ff(X,Y)$ consists of all $f\in\Hom_\Clu(X,Y)$ factoring
through $\Clu_0$.
\end{Proposition}

\begin{proof}
Since $\Hom_\Her(\Her_0,\Her_+)=0=\Ext^1_\Her(\Her_+,\Her_0)$ the
assertion follows for the first two cases. For the third case we use
that $\rad^\infty_\Her(\Her_0,\Her_0)$ equals zero. The last case is
covered by Proposition~\ref{prop:factor}.
\end{proof}

\begin{Corollary}\label{cor:recover}
The category $\Her$ can always be recovered from $\Clu$ as the
quotient of $\Clu$ by a two-sided ideal $\Ii$. This ideal is unique if
the Euler characteristic is non-zero. In the tubular case it only
depends on the choice of a tubular family $\Clu^{(q)}$ in $\Clu$, $q
\in \oQQ$. \proofend
\end{Corollary}

\section{Cluster tubes}
\label{sec:cluster_tubes}

\subsection*{Notation}

Let $\Tt$ be a cluster tube of rank $p$.  We denote the indecomposable
objects of $\Tt$ by $X^{(n)}_i$, $i\in\ZZ_p$ and $n$ an integer $\geq
1$, such that $\tau X^{(n)}_i=X^{(n)}_{i+1}$. Furthermore we choose
irreducible morphisms $\iota_i^{(n)}\colon X_i^{(n)}\rightarrow
X_i^{(n+1)}$ and $\pi_i^{(n)}\colon X_{i+1}^{(n+1)}\rightarrow
X_i^{(n)}$ satisfying
\begin{equation}
  \label{eq:comm-AR}
\iota^{(n-1)}_{i} \pi^{(n-1)}_{i}=\pi^{(n)}_{i} \iota^{(n)}_{i+1}
\end{equation}
for all $n\geq 1$ and $i\in\ZZ_p$. Here we used the convention
$X_i^{(0)}=0$ (and consequently $\iota_i^{(0)}=0$, $\pi_i^{(0)}=0$).
Whenever possible we will skip the indices and just write $\iota$ and
$\pi$. The situation is illustrated by the following figure.

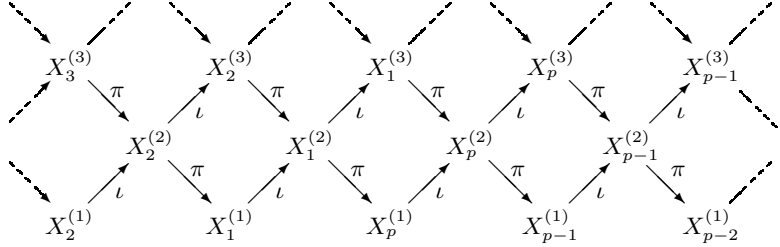
\begin{center}
  \begin{picture}(240,85)
  \put(0,0){
    \put(0,60){\HVCenter{\small $X_{3}^{(3)}$}}
    \put(0,0){\HVCenter{\small $X_{2}^{(1)}$}}
    \put(30,30){\HVCenter{\small $X_{2}^{(2)}$}}
    \put(60,60){\HVCenter{\small $X_{2}^{(3)}$}}
    \put(60,0){\HVCenter{\small $X_{1}^{(1)}$}}
    \put(90,30){\HVCenter{\small $X_{1}^{(2)}$}}
    \put(120,60){\HVCenter{\small $X_{1}^{(3)}$}}
    \put(120,0){\HVCenter{\small $X_{p}^{(1)}$}}
    \put(150,30){\HVCenter{\small $X_{p}^{(2)}$}}
    \put(180,60){\HVCenter{\small $X_{p}^{(3)}$}}
    \put(180,0){\HVCenter{\small $X_{p-1}^{(1)}$}}
    \put(210,30){\HVCenter{\small $X_{p-1}^{(2)}$}}
    \put(240,60){\HVCenter{\small $X_{p-1}^{(3)}$}}
    \put(240,0){\HVCenter{\small $X_{p-2}^{(1)}$}}
    \multiput(0,0)(60,0){4}{
      \put(7,7){\vector(1,1){15}}
      \put(7,53){\vector(1,-1){15}}
      \put(17,9){\small $\iota$}
      \put(15,47){\small $\pi$}
      \put(37,37){\vector(1,1){15}}
      \put(47,39){\small $\iota$}
      \put(45,17){\small $\pi$}
      }
    \multiput(0,0)(60,0){3}{
      \put(37,23){\vector(1,-1){16}}
      }
    \put(216,21){\vector(1,-1){14}}
    \multiput(0,60)(60,0){5}{
      \put(7,7){\qbezier(0,0)(3,3)(6,6)}
      \multiput(15,15)(3,3){3}{\qbezier(0,0)(0.75,0.75)(1.5,1.5)}
      \put(-7,7){\qbezier(-2,2)(-4,4)(-6,6)}
      \put(-7,7){\vector(1,-1){0.01}}
      \multiput(-15,15)(-3,3){3}{\qbezier(0,0)(-0.75,0.75)(-1.5,1.5)}
      }
    \put(240,0){
      \put(7,7){\qbezier(0,0)(3,3)(6,6)}
      \multiput(15,15)(3,3){3}{\qbezier(0,0)(0.75,0.75)(1.5,1.5)}
      }
    \put(0,0){
      \put(-7,7){\qbezier(-2,2)(-4,4)(-6,6)}
      \put(-7,7){\vector(1,-1){0.01}}
      \multiput(-15,15)(-3,3){3}{\qbezier(0,0)(-0.75,0.75)(-1.5,1.5)}
      }
    \put(0,60){
      \put(-7,-7){\qbezier(-2,-2)(-4,-4)(-6,-6)}
      \put(-7,-7){\vector(1,1){0.01}}
      \multiput(-15,-15)(-3,-3){3}{\qbezier(0,0)(-0.75,-0.75)(-1.5,-1.5)}
      }
    \put(240,60){
      \put(11,-11){\qbezier(0,0)(2,-2)(4,-4)}
      \multiput(17,-17)(3,-3){3}{\qbezier(0,0)(0.75,-0.75)(1.5,-1.5)}
      }
    }
  \end{picture}
\end{center}

\subsection*{Certain Yoneda products}

We shall show the following result concerning the Yoneda product.

\begin{Lemma}\label{lem:yp}
Let $\iota_i^{(n)}\colon X_i^{(n)}\rightarrow X_i^{(n+1)}$.
\begin{itemize}
\item[(a)]
  For any object $Z$ the linear map
$$
\Ext_\Her^1(X_i^{(n+1)},Z)\rightarrow
\Ext_\Her^1(X_i^{(n)},Z),\quad
\eta\mapsto\eta\iota_i^{(n)}
$$
is surjective.
\item[(b)]
  For $Z=X_j^{(n)}$ the linear map
$$
\Ext_\Her^1(X_{i}^{(n+1)},Z)\rightarrow
\Ext_\Her^1(X_{i}^{(n)},Z),\quad
\eta\mapsto\eta\iota_i^{(n)}
$$
is bijective.
\end{itemize}
\end{Lemma}

\begin{proof}
  To simplify notations write $X=X_i^{(n)}$, $Y=X_i^{(n+1)}$ and
  $\iota=\iota_i^{(n)}\colon X\rightarrow Y$.
  For (a) apply $\Hom(-,Z)$ to the short
  exact sequence
  $0\rightarrow X\xrightarrow{\iota}
    Y\xrightarrow{v}S\rightarrow 0$
  to get
  \begin{equation}
    \label{eq:ext-seq}
  \Ext^1(S,Z)\xrightarrow{v^\ast}
  \Ext^1(Y,Z)\xrightarrow{\iota^\ast}
  \Ext^1(X,Z)\rightarrow 0,
  \end{equation}
  where the last term $\Ext^2(S,Z)$ is zero because $\Her$ is
  hereditary.

  To see the injectivity in (b) apply Serre duality to
  \eqref{eq:ext-seq} to get
  $$
  \dual \Hom(Z,\tau S)\xrightarrow{\dual v^\ast}
  \dual \Hom(Z,\tau Y)\xrightarrow{\dual \iota^\ast}
  \dual \Hom(Z,\tau X).
  $$
  Since $Z$ has length $n$ whereas the
  uniserial object $Y$ has length $n+1$ the image of each morphism
  $f\in \Hom(Z,\tau Y)$ lies in the unique maximal submodule
  $\rad Y$ and thus $v\circ f=0$.
  This shows that $v^\ast\colon \Hom(Z,\tau Y)\rightarrow \Hom(Z,\tau S)$
  is zero.
\end{proof}

\subsection*{Automorphisms of cluster tubes}

To investigate automorphisms of cluster tubes from $\Clu$ we start
with a preliminary result.

\begin{Lemma} \label{lem:radical}
The infinite radical $\rad^\infty_{\Tt}$ of a cluster tube $\Tt$ equals $\Tt\cap\rad^\infty_\Clu$.
\end{Lemma}

\begin{proof}
It suffices to show that each $u\in \Tt\cap \rad^\infty_\Clu$ belongs to
$\rad^\infty_\Tt$. By \cite{LeMe} we may assume that $\Tt$ is a cluster
tube in $\Clu_0$. For this it is enough to show that each composition
$U_1 \xrightarrow{}E\xrightarrow{v}U_2$ belongs to $\rad^\infty_\Tt$,
where $U_1$, $U_2$ are indecomposables from $\Tt$ and $E$ lies in
$\Clu_+$. This uses that $\Clu_0$ consists of a family of pairwise
orthogonal cluster tubes. Note that $v$ belongs to $\Hom_\Her(E,U_2)$
since $\Ext^1_\Her(\Her_+,\Her_0)=0$. By Auslander-Reiten theory we
get a representation $v=u_n\cdots u_1v_n$ with radical morphisms $u_i$
from $\Tt$ for each integer $n\geq1$. This proves the claim.
\end{proof}

The following result is crucial for our analysis of automorphism
groups of the cluster category of a canonical
algebra.
\begin{Proposition}\label{Prop:exc_tube_iso}
 We fix an admissible triangulated structure on $\Clu$.
 Let $\Tt\subset\Clu$ be a cluster tube of rank $p\geq 1$. If $G\colon
 \Tt\rightarrow
  \Tt$ is an autoequivalence, sending induced triangles to exact
  triangles, and inducing the identity functor on $\Tt/\rad^\infty_\Tt$,
  then there is an isomorphism $\psi\colon 1 \xra{\sim} G$ of functors on
  $\Tt$ such that each $\psi_X\colon X\xra{\sim} G(X)$ has the form
  $\psi_X=1_X+\eta_X$, where $\eta_X$ is of degree one.
\end{Proposition}

\begin{proof}
  The proof is done in several steps.

  \noindent
  \textnormal{(1)}
  {\it $G$ is isomorphic to a functor $G'$ which also induces the
    identity in $\Tt/\rad^\infty_\Tt$ and additionally satisfies
  $G(\iota_i^{(n)})=\iota_i^{(n)}$ for any $n\geq 1$ and any
  $i\in\ZZ_p$.}

  By assumption we have $G(\iota_i^{(n)})=\iota_i^{(n)}+\xi_i^{(n)}$ for
  some $\xi_i^{(n)}\in \Ext_\Her^1(X_i^{(n)},\tau^- X_i^{(n+1)})$. We set
  $\eta_i^{(1)}=0$ and define inductively, using Lemma~\ref{lem:yp} (a),
  elements
  $\eta_i^{(n+1)}\in \Ext_\Her^1(X_i^{(n+1)},\tau^- X_i^{(n+1)})$ such that
  \begin{equation}
    \label{eq:def_f_inductyively}
    \eta_i^{(n+1)}\iota_i^{(n)}=\xi_i^{(n)}+\iota_i^{(n)}\eta_i^{(n)}.
  \end{equation}

  Next we define isomorphisms $\psi_i^{(n)}=1+\eta_i^{(n)}\colon
  X_i^{(n)}\rightarrow X_i^{(n)}$ yielding
  $G(\iota_i^{(n)})\psi_i^{(n)} = \psi_i^{(n+1)}\iota_i^{(n)}$.
  Therefore, setting $G'(X)=X$ for any object $X$ in $\Tt$ and
  $G'(f)=(\psi_{j}^{(n)})^{-1} \circ G(f)\circ \psi_i^{(m)}$ for any
  morphism $f\colon X_i^{(m)}\rightarrow X_j^{(n)}$ we obtain
  assertion (1).

  \noindent
  \textnormal{(2)}
  {\it The functor $G'$ also satisfies
    $G'(\pi_i^{(n)})=\pi_i^{(n)}$ for any $n\geq 1$ and any
    $i\in\ZZ_p$.}

  Recall that $G'(\pi_i^{(n)})=\pi_i^{(n)}+\xi_i^{(n)}$ for some
  $\xi_i^{(n)}\in \Ext_\Her^1(X_{i+1}^{(n+1)},\tau^- X_i^{(n)})$.  By
  induction on $n$ we shall show that $\xi_i^{(n)}=0$.  For $n=1$ it
  follows from $\pi_{i}^{(1)}\iota_{i+1}^{(1)}=0$ that
  $0=G'(\pi_{i}^{(1)}\iota_{i+1}^{(1)})=
  \left(\pi_{i}^{(1)}+\xi_{i}^{(1)}\right) \iota_{i+1}^{(1)}$. Hence
  by Lemma~\ref{lem:yp} (b), we obtain $\xi_{i}^{(1)}=0$.  Assuming
  inductively that $G'(\pi_i^{(n-1)})=\pi_i^{(n-1)}$ for any $i$, we
  can apply $G'$ to the identity~\eqref{eq:comm-AR} and obtain
  $G'(\pi_{i}^{(n)}\iota_{i+1}^{(n)})=\pi_{i}^{(n)}\iota_{i+1}^{(n)}$.
  Therefore
  $\pi_{i}^{(n)}\iota_{i+1}^{(n)}=G(\pi_{i}^{(n)})\iota_{i+1}^{(n)}=
  (\pi_{i}^{(n)}+\xi_i^{(n)})\iota_{i+1}^{(n)}$ which implies
  $\xi_i^{(n)}\iota_{i+1}^{(n)}=0$. By Lemma~\ref{lem:yp} (b) we get
  $\xi_i^{(n)}=0$.

  \noindent
  \textnormal{(3)}
  {\it The functor $G'$ is the identity functor.}

  Let $\eta\in\Ext^1_\Her(X_a^{(r)},X_b^{(s)})$ be given by the
  sequence $ 0 \rightarrow X_b^{(s)}\xrightarrow{f} E\xrightarrow{g}
  X_a^{(r)}\rightarrow 0$, which gives rise to the induced triangle
  $\Delta$ forming the upper row of the following diagram. Since $G$
  and therefore also $G'$ is exact on $\Delta$, $G'(\Delta)$ is a
  triangle again.  By axiom \cite[(TR3)]{Ve} there exists a morphism
  $\zeta\colon X_a^{(r)}\rightarrow X_a^{(r)}$ making the following
  diagram commutative.
  \begin{center}
    \begin{picture}(210,44)
      \put(20,5){
        \put(-20,0){\RVCenter{\small $G'(\Delta)\colon $}}
        \put(-20,35){\RVCenter{\small $\Delta\colon $}}
        \multiput(0,0)(0,35){2}{
          \put(0,-4){\HBCenter{\small $X_b^{(s)}$}}
          \put(70,-4){\HBCenter{\small $E$}}
          \put(140,-4){\HBCenter{\small $X_a^{(r)}$}}
          \put(210,-4){\HBCenter{\small $X_b^{(s)}[1]$}}
          \put(14,0){\vector(1,0){48}}
          \put(79,0){\vector(1,0){48}}
          \put(153,0){\vector(1,0){39}}
          }
        \multiput(0,0)(70,0){2}{
          \multiput(-1,8)(2,0){2}{\line(0,1){17}}
          }
        \multiput(209,8)(2,0){2}{\line(0,1){17}}
        \multiput(140,27)(0,-4){5}{\line(0,-12){2}}
        \put(140,14){\vector(0,-1){6}}
        \put(35,-3){\HTCenter{\small $f$}}
        \put(35,40){\HBCenter{\small $f$}}
        \put(105,-3){\HTCenter{\small $g$}}
        \put(105,40){\HBCenter{\small $g$}}
        \put(173,-2){\HTCenter{\small $G'(\eta)$}}
        \put(175,40){\HBCenter{\small $\eta$}}
        \put(143,17.5){\LVCenter{\small $\zeta$}}
}
\end{picture}
\end{center}
Write $\zeta=\zeta_0+\zeta_1$, the decomposition into
different degrees. Then $g=g_0=(\zeta g)_0=\zeta_0 g$ implies
$\zeta_0=1$ since $g$ is surjective. Therefore
it follows from $\eta=G'(\eta)\zeta$ that
$\eta=G'(\eta)+G'(\eta)\zeta_1=G'(\eta)$,
where the last equation follows from Lemma~\ref{lem:radical}  and
the fact that the composition of two morphisms of
degree one is zero. Hence the result.
\end{proof}


\section{Lifting of automorphisms}
\label{sec:lifting}

\subsection*{Non-zero prolongation}
We shall need the following results.

\begin{Lemma}
  \label{non-zero-composition-1}
  For any non-zero $\eta\in\Ext_\Her^1(X,\tau^- Y)$ there exists a morphism
  $h\in\Hom_\Her(\tau^- Y ,\tau X)$ such that $h\eta\neq0$.
\end{Lemma}

\begin{proof}
  Observe that $\eta\in\Ext_\Her^1(X,\tau^- Y)$ is given by a
  non-split short exact sequence $ \eta\colon 0\rightarrow \tau^-
  Y\xrightarrow{a} F \xrightarrow{b} X\rightarrow0$. Therefore $b$
  factors through $b'$ in the almost-split sequence $
  \varepsilon_X\colon 0\rightarrow \tau X\xrightarrow{a'}
  E\xrightarrow{b'} X\rightarrow0, $ that is, $b=b'c $ for some
  $c\in\Hom_\Her(F,E)$ and consequently $c a$ factors through $a'$,
  say $c a=a'h$ for some $h\in\Hom_\Her(\tau^-Y,\tau X)$.  Therefore
  $h \eta=\varepsilon_X$ is non-zero as an almost-split sequence.
\end{proof}

\begin{Lemma}\label{non-zero-composition-2}
  Assume $f\colon X\rightarrow Y$ is a non-zero morphism in $\Her_+$.
  Then there exists a morphism $h\colon Y\ra Z$ in $\Her_+$ such that
  $hf\neq0$ and $\Ext^1_\Her(X,\tau^- Z)=0=\Ext^1_\Her(Y,\tau^- Z)$.
\end{Lemma}

\begin{proof}
  By \cite[Cor.~2.7]{GeLe} there exists an embedding
  $Y\hookrightarrow\bigoplus_{i=1}^rL_i$ into a direct sum of line
  bundles $L_i$. For each integer $n\geq0$ we further obtain
  embeddings $L_i\ra L_i^{(n)}$ into line bundles $L_i^{(n)}$ with
  degree (slope) $\mu(L_i^{(n)})\geq \mu(L_i)+n$ (in the language of
  \cite{GeLe} we may take $L_i^{(n)}=L_i(n\vec{c})$). We thus obtain
  an embedding $h\colon Y\ra Z$ with $Z=\bigoplus_{i=1}^r L_i^{(n)}$.
  Clearly $hf\neq0$. By means of line bundle filtrations for $X$ and
  $Y$ the remaining assertions now follow as in \cite[(S15)]{LePe} if
  $n$ is sufficiently large.
\end{proof}

\subsection*{Autoequivalences of $\Clu$ fixing all objects}

The following two results are the key ingredients to determine the
automorphism group of $\Clu$.
\begin{Proposition}
  \label{prop:aut_fix_obj}
  We fix an admissible triangulated structure on $\Clu$.  If an
  autoequivalence $G$ of $\Clu$ sends induced triangles to exact
  triangles and fixes all objects then $G$ is isomorphic to the
  identity functor.
\end{Proposition}

\begin{proof}
Let $\Ii$ be the two-sided ideal of morphisms of degree one of
$\Clu$. Since $G$ is the identity on objects it follows from the
description of $\Ii$ in Proposition~\ref{prop:ideal} that
$G(\Ii)=\Ii$. Hence $G$ induces an autoequivalence $\sigma$ of
$\Clu/\Ii=\Her$.

Observe that $G$ is isomorphic to an autoequivalence $G'$ which
satisfies the following two properties:
  \begin{itemize}
  \item[\textnormal{(i)}] $G'(X)=X$ for any object
  $X\in\Clu$,
  \item[\textnormal{(ii)}] for each morphism $f$ of degree zero there exists
    a morphism  $\eta_f$ of degree one such that $G'(f)=f+\eta_f$.
  \end{itemize}

Indeed, by \cite[Prop.~2.1]{LeMe}, the autoequivalence $\sigma$ is
  isomorphic to the identity functor, say $\varphi\colon
  \sigma\xrightarrow{\sim}
  1_\Her$.  Clearly, $\varphi_X$ is a degree zero isomorphism in
  $\Clu$ for each
  object $X\in\Clu$ and therefore we can define  $G'$ on morphisms by setting
  $G'(f)=\varphi_Y\circ G(f)\circ\varphi_X^{-1}$ for
  $f\colon X\rightarrow Y$.

By Proposition~\ref{Prop:exc_tube_iso} we get an isomorphism
$\psi\colon 1_{\Clu_0}\ra G'|_{\Clu_0}$, where the degree zero part of
$\psi_X$ is the identity on $X$. Changing $G'$ by means of $\psi$ we
may additionally assume that the restriction of $G'$ to $\Clu_0$ is
the identity functor. Due to the special form of $\psi$, see
Proposition~\ref{Prop:exc_tube_iso}, the functor $G'$ still keeps
property~(ii).

Simplifying notation we write $G$ instead of $G'$. We show now that
$G(f)=f$ for each $f\colon X\ra Y$ in the remaining cases: (a)
$X\in\Clu_+$, $Y\in\Clu_0$, (b) $X\in\Clu_0$, $Y\in\Clu_+$ and (c)
$X\in\Clu_+$, $Y\in\Clu_+$.

In case (a) this follows from the fact that there are no morphisms of
degree one from $\Clu_+$ to $\Clu_0$.

To prove case (b) we assume $G(f)\neq f$. Since there are no morphisms
of degree zero from $X$ to $Y$, Lemma~\ref{non-zero-composition-1}
yields a morphism $h\colon Y\ra \tau X$ such that $h(f-G(f))\neq0$.
Then $G(hf)=hf$ since $G$ is the identity on $\Clu_0$. On the other
hand $G(hf)=hG(f)$ by case (a), yielding the contradiction
$h(f-G(f))=0$.

Concerning (c) we first assume that $f$ has degree one. By
Proposition~\ref{prop:factor} we get a factorization $f=\eta h$ for
some $h\in\Hom_\Her(X,Z)$ and $\eta\in\Ext^1_\Her(Z,\tau^-Y)$ for some
$Z\in\Clu_0$. We thus get $G(f)=f$ by (a) and (b). Next we assume that
$f$ is a degree zero morphism with $f-G(f)\neq0$. By
Lemma~\ref{non-zero-composition-2} we find $Z\in \Clu_+$ and a
morphism $h\colon Y\ra Z$ such that $h(f-G(f))\neq0$ and
$\Hom_\Clu(X,Z)_1=0=\Hom_\Clu(Y,Z)_1$. Then again $G(hf)=hf$ and
$G(h)=h$, hence $h(f-G(f))=0$, a contradiction.  We have shown that
$G$ is the identity functor on $\Clu$.
\end{proof}

\subsection*{A lifting property}

We first observe that each exact autoequivalence $u$ of $\Der(\Her)$
induces canonically an autoequivalence $u_\Clu$ of $\Clu$ which makes
the diagram
\begin{center}
  \begin{picture}(100,48)
    \put(0,1){
    \multiput(0,5)(100,0){2}{\HBCenter{$\Clu$}}
    \multiput(0,35)(100,0){2}{\HBCenter{$\Der(\Her)$}}
    \multiput(0,31)(100,0){2}{\vector(0,-1){17}}
    \put(8,8){\vector(1,0){84}}
    \put(20,38){\vector(1,0){60}}
    \put(-3,23){\RVCenter{$\pi$}}
    \put(103,23){\LVCenter{$\pi$}}
    \put(50,5){\HTCenter{$u_\Clu$}}
    \put(50,40){\HBCenter{$u$}}
    }
  \end{picture}
\end{center}
commutative and sends induced triangles to induced triangles. The
setting induces a canonical homomorphism
$\pi_{\ast}\colon\Aut(\Der(\Clu))\ra \Aut(\Clu_{\ind})$, $u\mapsto
u_\Clu$, where $\Aut(\Clu_{\ind})$ denotes the semigroup of
isomorphism classes of autoequivalences of the category $\Clu$,
commuting with the translation [1] and preserving induced triangles.
It follows from the next proposition that $\Aut(\Clu_{\ind})$ actually
is a group, called the automorphism group of $\Clu_\ind$.

\begin{Theorem}\label{thm:lifting}
  Let $\Clu$ be the cluster category of $\Her$ equipped with an
  admissible triangulated structure. Then each autoequivalence $G$ of
  $\Clu$, sending induced triangles to exact triangles, lifts to an
  exact autoequivalence $u$ of the derived category $\Der(\Her)$, that
  is, $G$ is isomorphic to $u_\Clu$.  In particular, $G$ sends induced
  triangles to induced triangles.
\end{Theorem}

\begin{proof}
  We recall that $\Clu$ has a unique tubular family if the
  Euler characteristic of $\Her$ is non-zero and therefore
  $G(\Clu_0)=\Clu_0$. In the tubular case we can assume the same,
  changing $G$ by an automorphism of $\Der(\Her)$, if necessary: Note
  that the automorphism group of $\Der(\Her)$ acts transitively on the
  set of tubular families of $\Clu$ by means of $\pi_*$. Since
  $G^{-1}(\Clu_0)=\Clu^{(q)}$ for some $q$, we find an automorphism
  $u$ of $\Der(\Her)$ with $u_\Clu(\Clu_0)=\Clu^{(q)}$, and consider
  $Gu_\Clu$ instead of $G$.

The property $G(\Clu_0)=\Clu_0$ now implies that $G(\Ii)=\Ii$ for the
ideal $\Ii$ of degree one morphisms of $\Clu$. Consequently $G$
induces an autoequivalence $\gamma$ of $\Her=\Clu/\Ii$. We then extend
$\gamma$ to the exact autoequivalence $u=\Der(\gamma)$ of $\Der(\Her)$
and consider $u_\Clu$. By construction $G\,u_\Clu^{-1}$ is exact on
induced triangles and fixes the objects of $\Clu$, hence $G$ is
isomorphic to $u_\Clu$ by Proposition~\ref{prop:aut_fix_obj}, proving
the claim.
\end{proof}

\begin{Corollary} \label{cor:surjective}
The homomorphism $\pi_*\colon \Aut(\Der(\Her))\ra \Aut(\Clu_\ind)$ is
surjective.
\end{Corollary}

By $\Clu_\kel$ we denote the cluster category $\Clu$, equipped with
the triangulated structure defined in \cite{Ke}. We are going to show
that exact autoequivalences of $\Der(\Her)$ induce exact
autoequivalences of the category $\Clu_\kel$.

\begin{Lemma}\label{lem:standard} Each
  exact autoequivalence $G$ of $\Der(\Her)=\Der(A)$, $A$ canonical, is
  standard in the sense of~\cite[Def.~3.4]{Ric2}.
\end{Lemma}

\begin{proof}
  It follows from~\cite[Cor.~3.5]{Ric2} that there is a standard
  autoequivalence $G'$ such that $G(X)\simeq G'(X)$ for all objects $X$
  in $\Der(A)$. From~\cite[Prop.~2.1]{LeMe} we get $G\simeq G'$,
  and hence $G$ is standard itself.
\end{proof}

\begin{Proposition}
  An autoequivalence of $\Clu$ is exact in $\Clu_\kel$ if and only if
  it is exact on induced triangles. Moreover,
  $\Aut(\Clu_\kel)=\Aut(\Clu_\ind)$.
\end{Proposition}

\begin{proof}
  We first note that each autoequivalence of $\Der(A)$ which is
  standard in the sense of~\cite[Def.~3.4]{Ric2} is also standard in
  the sense of~\cite[9.8]{Ke}. By Lemma~\ref{lem:standard} Keller's
  result \cite[9.4]{Ke} implies that each exact autoequivalence of
  $\Der(\Her)$ induces an exact autoequivalence of $\Clu_\kel$. Both
  assertions now follow from Theorem~\ref{thm:lifting}.
\end{proof}

\begin{Remark}
  Each autoequivalence $G$ of $\Clu$ which is exact with respect to an
  admissible triangulated structure belongs to
  $\Aut(\Clu_\ind)=\Aut(\Clu_\kel)$. It is not clear whether the
  converse inclusion also holds.
\end{Remark}

\section{The automorphism group of the cluster category}
\label{sec:automorphisms}

We now continue our study of the automorphism group $\Aut(\Clu)$ where
the cluster category $\Clu=\Clu(\Her)$ is equipped with the
triangulated structure due to Keller~\cite{Ke}. Note that $F$
  defines a central element in $\Aut(\Der(\Her))$.

 \begin{Theorem}\label{Thm:Aut1}
The automorphism group $\Aut(\Clu)$ is
canonically isomorphic to the quotient $\Aut(\Der(\Her))/\spitz{F}$.
\end{Theorem}
\begin{proof}
  The homomorphism $\pi_*\colon \Aut(\Der(\Her))\rightarrow
  \Aut(\Clu)$, $u\mapsto u_\Clu$ is surjective by
  Corollary~\ref{cor:surjective}.  Clearly $F$ lies in the kernel of
  $\pi_*$. Suppose now that $u\in\Aut(\Der(\Her))$ satisfies
  $u_\Clu=1$.  Then, for any indecomposable object $X\in\Der(\Her)$,
  we must have $u(X)\simeq F^{n_X}(X)$ for some integer $n_X$. In
  particular, $u(X)$ belongs to $\Her_0[n_X]$ for each indecomposable
  $X$ from $\Her_0$. Since $u(\Her_0)$ is a tubular family in
  $\Der(\Her)$ it then follows that $u(\Her_0)=\Her_0[n]$ for a fixed
  integer $n$. By~\cite[Prop.~4.2 and 6.2]{LeMe} this in turn implies
  that $u(\Her)=\Her[n]$. Hence $u$ is isomorphic to $F^n$ on objects
  of $\Der(\Her)$, implying that $u$ is isomorphic to $F^n$ as a
  functor by \cite[Prop.~2.1]{LeMe}.
\end{proof}

\subsection*{The tubular situation}

We assume that $\Her$ is tubular and first give an invariant
description of the set $\oQQ$ of slopes for $\Her$. We define $\WW$ as
the set of all rank one direct summands of the free abelian group
$R=\rad\Groth(\Her)$ of rank two, where $R$ consists of all elements
of the Grothendieck group fixed under the automorphism induced by
$\tau$. We consider $R$ to be equipped with the Euler form.

Let $p$ be the least common multiple of the weights of $\Her$. Fixing
$w$ in $\WW$ the (additive closure of the) subcategory consisting of
all indecomposable objects $X=\pi(Y)$ of $\Clu$ such that
$\sum_{j=1}^p[\tau^jY]$ belongs to $w$ is a tubular family
$\Clu^{(w)}$ in $\Clu$ and each tubular family in $\Clu$ has this
form. This follows from~\cite[Thm.~4.6]{LeMe2}. In this way we get a
bijection $\WW\xrightarrow{\sim}\oQQ$. From \cite{LeMe} we know that
the group $M$ of $\ZZ$-linear automorphism of $R$ preserving the Euler
form is isomorphic to $\SL_2(\ZZ)$. We consider $M/\{\pm1\}$ as the
automorphism group of $\WW$ and hence $\Aut(\WW)\simeq \PSL_2(\ZZ)$.
We call $\WW$ the \emph{rational circle}.

\begin{Lemma}\label{lem:circle}
  There is a natural surjective homomorphism $\mu\colon\Aut(\Clu)\ra
  \Aut(\WW)$ such that $\rho=\mu(G)$ satisfies
  $G(\Clu^{(w)})=\Clu^{(\rho(w))}$ for each automorphism $G$ of
  $\Clu$.
\end{Lemma}

\begin{proof}
By Corollary~\ref{cor:surjective} an automorphism $G$ of $\Clu$ lifts
to an automorphism $u$ of $\Der(\Her)$. Since $u$ induces an
automorphism of the Grothendieck group, preserving the Euler form, it
induces an automorphism $u_\WW$ of $\WW$. Note that the automorphism
$F=\tau^{-}[1]$ induces the identity on $\WW$. By
Theorem~\ref{Thm:Aut1} two liftings of $G$ differ by a power of $F$,
hence induce the same element $G_\WW=u_\WW$ of $\Aut(\WW)$. This
defines the homomorphism $\mu$ satisfying the above formula. By
\cite[Thm.~6.3]{LeMe} surjectivity of $\mu$ follows.
\end{proof}

\begin{Remark}
  For the preceding arguments it is essential to relate the cluster
  category with the Grothendieck group of the derived category
  $\Der(\Her)$ by means of a lifting property. The Grothendieck group
  $\Groth(\Clu)$ of the cluster category, investigated in~\cite{BKL},
  cannot be used for the present purpose since essential information
  on the radical is lost under the canonical projection
  $\Groth(\Her)\ra \Groth(\Clu)$.
\end{Remark}

We recall from~\cite{LeMe} that the automorphism group $\Aut(\XX)$ of
a weighted projective line $\XX$ consists of all members of
$\Aut(\Her)$ fixing the structure sheaf. This group is finite if $\XX$
has at least three weights, in particular if $\Her$ is tubular.
Further the subgroup $\Pic_0(\XX)$ of
$\Aut(\Her)$ consists of all shift functors $E\mapsto E(\vec{x})$ of
$\Her$ which preserve slopes. This group is always finite abelian.

\begin{Proposition} \label{prop:tub_aut}
Assume $\Her$ is tubular. Then there is an exact sequence
  $$1\rightarrow \Pic_0(\XX)\rtimes
  \Aut(\XX)\rightarrow\Aut(\Clu)\xrightarrow{\mu} \Aut(\WW)\rightarrow 1.
  $$
\end{Proposition}
\begin{proof}
Indeed, we have the following commutative diagram
\begin{center}
\begin{picture}(310,128)
\put(0,15){
  \put(0,12){
  \put(0,-3){
    \multiput(180,95)(80,0){2}{\HBCenter{$1$}}
    \multiput(180,-25)(80,0){2}{\HBCenter{$1$}}
  \multiput(10,0)(0,35){2}{\HBCenter{$1$}}
  \multiput(80,0)(0,35){2}{\HBCenter{$\Pic_0(\XX)\rtimes\Aut(\XX)$}}
  \put(180,0){\HBCenter{$\Aut(\Clu)$}}
  \put(180,35){\HBCenter{$\Aut(\Der(\Her))$}}
  \put(255,35){\HBCenter{$\Aut(\YY)$}}
  \put(255,0){\HBCenter{$\Aut(\WW)$}}
  \put(180,70){\HBCenter{$\spitz{F}$}}
  \put(260,70){\HBCenter{$\spitz{t}$}}
  \multiput(300,0)(0,35){2}{\HBCenter{$1$}}
  }
  \multiput(15,0)(0,35){2}{\vector(1,0){22}}
  \put(123,0){\vector(1,0){38}}
  \put(123,35){\vector(1,0){27}}
  \put(199,0){\vector(1,0){36}}
  \put(210,35){\vector(1,0){26}}
  \put(275,0){\vector(1,0){20}}
  \put(274,35){\vector(1,0){21}}
  \multiput(79,7)(2,0){2}{\line(0,1){19}}
  \multiput(180,0)(80,0){2}{%
    \multiput(0,0)(0,35){2}{%
      \put(0,28){\vector(0,-1){21}}
      }
    \multiput(0,-7)(0,95){2}{\vector(0,-1){11}}
    }
  \put(190,70){\vector(1,0){61}}
  \put(220,71){\HBCenter{$\sim$}}
}
}
\end{picture}
\end{center}
where the vertical exact sequence on the right is taken from the proof
of \cite[Thm.~5.1]{LeMe}. By Theorem~\ref{Thm:Aut1} and
\cite[Thm.~6.3]{LeMe} the central vertical and horizontal sequences
are also exact, yielding the claim.
\end{proof}

\subsection*{Comparison of $\mathrm{Aut}(\mathcal{H})$ and $\mathrm{Aut}(\mathcal{C})$}

There is a natural homomorphism $j\colon\Aut(\Her)\ra \Aut(\Clu)$
sending an automorphism $u$ of $\Her$ to the automorphism $u_\Clu$ of
$\Clu$ induced by $\Der(u)$. By Theorem~\ref{Thm:Aut1} $j$ is
injective. In the following we identify $\Aut(\Her)$ with its image in
$\Aut(\Clu)$.

\begin{Theorem} \label{thm:AutC-AutH}
$\Aut(\Her)$ consists of those automorphisms $G$ of $\Clu$ with
$G(\Clu_0)=\Clu_0$. Moreover,

\textnormal{(i)} If $\chi_\Her\neq0$ we have $\Aut(\Her)=\Aut(\Clu)$.

\textnormal{(ii)} If $\Her$ is tubular, we have a canonical bijection
$ \nu\colon\Aut(\Clu)/\Aut(\Her)\xrightarrow{\sim}\WW, $ where
$G\cdot\Aut(\Her)$ is mapped to $w$ with $G(\Clu_0)=\Clu^{(w)}$. The
bijection is compatible with the natural left $\Aut(\Clu)$-actions on
both sets.
\end{Theorem}

\begin{proof}
  The first assertion and (i) follow from
  Theorem~\ref{thm:lifting}. Concerning (ii) we note that
  $\Aut(\Clu)$ acts transitively on $\WW$ by means of
  Lemma~\ref{lem:circle}. The bijectivity follows.
\end{proof}

\begin{Remark}
  With a little extra work, concerning
  Proposition~\ref{Prop:exc_tube_iso} and using~\cite[Cor.~3.3]{Kus3}
  instead of~\cite{LeMe}, the results of our paper extend to the case
  of an arbitrary base field as long as canonical algebras in the
  sense of Ringel~\cite{Ri} or weighted projective lines are
  concerned. In the more general situation, dealing with canonical
  algebras in the sense of~\cite{Ri2}, equivalently with categories of
  coherent sheaves on exceptional curves~\cite{Le}, the uniqueness
  result Proposition~\ref{prop:aut_fix_obj} does not carry over
  by~\cite{Kus2}, affecting our proofs of Theorem~\ref{thm:lifting}
  and subsequent results (Theorem~\ref{Thm:Aut1},
  Theorem~\ref{thm:AutC-AutH}).
\end{Remark}


\section{The tilting graph}

In this section $\Her$ denotes the category of coherent sheaves on a
weighted projective line and $\Clu$ the corresponding cluster
category, equipped with an admissible triangulated structure. We
are proving in this section that the tilting (or exchange) graphs for
$\Her$ and $\Clu$ agree and, moreover, are connected if the Euler
characteristic is non-negative.

 We start with a result due to H\"{u}bner~\cite[Prop.~5.14]{Huebner}.
\begin{Proposition}
Let $T=E\oplus U$ be a basic tilting object in $\Her$ with $E$
indecomposable. Then there exists exactly one exceptional object $E^*$
in $\Her$ such that $E^*$ is not isomorphic to $E$ and $T'=E^*\oplus
U$ is also a tilting object.

Moreover, exactly one of the spaces $\Ext^1_\Her(E,E^*)$ and
$\Ext^1_\Her(E^*,E)$ is non-zero, and then one-dimensional over
$k$.~\qed
\end{Proposition}

The object $T'$ is called \emph{mutation} of $T$ in $E$.  Moreover, if
a tilting object $T'$ can be obtained by a sequence of mutations from
a tilting object $T$, then we say that $T'$ and $T$ are
\emph{connected by mutations}. If $E$ is a source (resp.\ sink) in the
quiver of $\End (T)$ then the mutation is given by APR-tilt
(resp.\ APR-cotilt), see~\cite{APR}.

The \emph{tilting graph} $\graph{\Her}$ of $\Her$ has as vertices the
isomorphism classes of basic (or multiplicity-free) tilting objects
$T$ of $\Her$. Two vertices, represented by basic tilting objects $T$
and $T'$ are connected by an edge if and only if they differ by
exactly one indecomposable summand. By Proposition~\ref{prop:CandH}
this graph may be identified with the tilting (or exchange) graph
$\graph{\Clu}$ of the cluster category $\Clu=\Clu(\Her)$,
compare~\cite{5clu}. Note that this graph is hence independent of the
choice of the admissible triangulated structure.

\begin{Lemma} \label{lem:trivial}
Let $T$ be a cluster-tilting object in $\Clu$ and $\comp{T}$ its
connected component in $\graph{\Clu}$. If $\si$ is an automorphism of
$\Clu$ such that $\si(T)$ belongs to $\comp{T}$ then
$\si^{\pm1}(\comp{T})=\comp{T}$.
\end{Lemma}
\begin{proof}
Note that $\si(\comp{T})$ is connected having a non-empty intersection
with $\comp{T}$.
\end{proof}
\newcommand{\can}{T_{can}}
\newcommand{\squid}{T_{sq}}
\newcommand{\epi}{\twoheadrightarrow}
\newcommand{\wing}{\mathcal{W}}
By a \emph{wing} in $\Her_0$ we understand the full subcategory
$\wing$ consisting of all objects having a finite filtration whose
factors belong to a \emph{proper $\tau$-segment} $S,\tau S,\tau^2
S,\ldots,\tau^{r-1}S$ of simple objects of an exceptional tube of
$\Her_0$. Proper means that not all simple objects of the tube belong
to the segment. By construction each wing $\wing$ with $r$ simples is
equivalent to the $k$-linear representations of the linear quiver
$1\ra2\ra3\ra\cdots\ra r$. Each tilting object $B$ in $\wing$ is called a
\emph{branch}. Such a branch has $r$ indecomposable (non-isomorphic)
direct summands and always contains a simple object and also the
\emph{root} $R$ of $\wing$, defined as the unique indecomposable of
maximal length $r$.

Assume $\Her$ has weight type $(p_1,\ldots,p_t)$. In the notation of
\cite{GeLe}, the line bundles $\Oo(\vx)$, with $0\leq\vx\leq\vc$,
form a tilting object $\can$ in $\Her$ whose endomorphism algebra is
the canonical algebra attached to $\Her$. We call $\can$ the
\emph{standard canonical configuration} and
$\Oo(\vx_i),\Oo(2\vx_i),\ldots,\Oo((p_i-1)\vx_i)$, $i=1,\ldots,t$, the
\emph{$i$-th arm} of $\can$.  Note that $\Her$ has exactly $t$
exceptional tubes consisting of sheaves of finite length. In the
$i$-th exceptional tube of rank $p_i$ there is exactly one simple
object $S_i$ satisfying $\Hom_\Her(\Oo,S_i)\neq0$. Moreover, in the
same tube there exists a sequence of exceptional objects and
epimorphisms
$$
B_i \colon S_i^{[p_i-1]}\epi S_i^{[p_i-2]}\epi \cdots \epi
S_i^{[1]}=S_i
$$
where $S_i^{[j]}$ has length $j$ and top $S_i$. The direct sum of
$\Oo$, $\Oo(\vc)$ and all the $S_i^{[j]}$, $i=1,\ldots,t$;
$j=1,\ldots,p_i-1$, forms another tilting object of $\Her$, called the
\emph{standard squid $\squid$}. Further $B_i$ is called the
\emph{$i$-th branch} of $\squid$.

We recall that the \emph{Picard group} $\Pic(\XX)$ of $\XX$ (or
$\Her$) is the subgroup of the automorphism group of $\Her$ consisting
of all shift functors $E\mapsto E(\vx)$, see~\cite{GeLe}.

\begin{Proposition}[\cite{Huebner}, Prop.~5.30]   \label{prop:Huebner}
The standard squid $\squid$ belongs to the connected  component
$\Delta=\comp{\can}$ of $\can$ in $\graph{\Her}$. Moreover $\Delta$ is
preserved under the operations of $\Pic(\XX)$, in particular under
Auslander-Reiten translation.
\end{Proposition}
\begin{proof}
  For the convenience of the reader we include a proof. To obtain
  a sequence of mutations transforming $\can$ into
  $\squid$ we consider the following part 
  $$\Oo\ra\Oo(\vx_h)\ra\Oo(2\vx_h)\ra\dots\ra\Oo((p_h-1)\vx_h)\ra\Oo(\vc)$$
  of $\can$.  Forming successively mutations in $\Oo
  ((p_h-1)\vx_h),\dots,\Oo(\vx_h)$ we replace first
  $\Oo((p_h-1)\vx_h)$ by $S_h$, then $\Oo((p_h-2)\vx_h)$ by
  $S_h^{[2]}$, and so on, and finally $\Oo(\vx_h)$ by $S_h^{[p_h-1]}$.
  Dealing with all the arms, we thus obtain $\squid$ from $\can$.

We next show that $\Delta$ is preserved under shift by $-\vx_h$, where
$h=1,\ldots,t$.  The branch $B_h$ of $\squid$ forms a tilting object
in the \emph{wing} $\Ww_h$ generated by the $\tau$-segment $S_h, \tau
S_h, \ldots, \tau^{p_i-2} S_h$. Since $\Ww_h$ is equivalent to the
category of $k$-linear representations of the quiver $1\ra 2\ra
\cdots\ra (p_h-1)$, the tilting graph of $\Ww_h$ is connected.
Without changing the component of $\squid$, we may thus replace the
linear branch $B_h$ by the branch $S_h^{[p_h-1]},\tau
S_h^{[p_h-2]},\ldots,\tau S_h$, while keeping the other indecomposable
summands of $\squid$.

Next we form a sequence of four mutations, replacing first
$S_h^{[p_h-1]}$ by $\Oo(\vc-\vx_h)$, next $\Oo$ by $\Oo(2\vc -\vx_h)$,
then $\Oo(\vc)$ by $\tau S_h^{[p_h-1]}$ and finally $\Oo(2\vc-\vx_h)$
by $\Oo(-\vx_h)$. By way of example we verify the second last step,
which is the only one where the mutation is not given by an APR-tilt
or APR-cotilt.  In the endomorphism ring of the tilting object $T$,
consisting of $\Oo(\vc-\vx_h),\,\Oo(\vc),\,\Oo(2\vc-\vx_h)$, $B_j$
($j\neq h$) and $B'_h =\{\tau S_h,\dots,\tau S^{[p_h -2]}\}$, there is
a unique irreducible map (up to scalars) starting in $\Oo(\vc)$,
namely $\Oo(\vc)\stackrel{x_h^{p_h-1}}\longrightarrow\Oo(2\vc-\vx_h)$.
Passage to the cokernel yields $\tau S_h^{p_h-1}$, the exceptional
object replacing $\Oo(\vc)$.

  To sum up, we have shown that $\squid(-\vx_h)$ lies in the component
  of $\squid$. By Lemma~\ref{lem:trivial} this shows that
  $\comp\squid$ is stable under the operations of $\Pic(\XX)$.
\end{proof}

For a tilting object $T$ we write $T=T_+\oplus T_0$ for the
  summands $T_+\in \Her_+$ and $T_0\in\Her_0$ and we call $T_0$ the
  \emph{torsion part} of $T$.

\begin{Proposition} \label{prop:reduction}
Let $T$ be a tilting object in $\Her$ with non-zero torsion part
$T_0$. Then there exist tilting objects $T'$ and $T''$ in the
component $\comp{T}$ where $T'$ lies in $\Her_+$ and where  the
torsion part of $T''$ is an exceptional simple object.
\end{Proposition}

\begin{proof}
  By \cite{LeMe3} the torsion part of $T$ has a decomposition
  $T_0=\bigoplus_{j\in J}B_j$ where each $B_j$ is a branch in a wing
  $\wing_j$ and, moreover, the wings $\wing_j$ are pairwise Hom- and
  Ext-orthogonal. Fixing an index $j$, we may assume by a sequence of
  mutations inside $\wing_j$ that the root $R_j$ of $\wing_j$ is a
  sink in the endomorphism ring of $T$. Mutation (APR-cotilt) at $R_j$
  then replaces $R_j$ by the kernel term $R_j^*$ of an exact sequence
  $0\ra R_j^*\ra \bigoplus_{i\in K} T_i \ra R_j\ra 0$, where each
  $T_i$ is an indecomposable summand of $T$ different from $R_j$. Due
  to connectedness of $\End(T)$, at least one $T_i$, $i\in K$, has
  positive rank. It follows that $R_j^*$ has positive rank. The claim
  now follows by induction.
\end{proof}

\begin{Proposition} \label{prop:domestic}
  Assume that $\chi_\Her>0$. For any two tilting bundles $T$ and $T'$
  in $\Her$ there exists a mutation sequence of tilting bundles
  linking $T$ and $T'$.
\end{Proposition}
\begin{proof}
  Note that the Auslander-Reiten quiver $\Ga$ of $\Her_+$ has shape
  $\ZZ\Delta$ where $\Delta$ is extended Dynkin.  \emph{Step 1.} We
  show first that each tilting bundle admits a mutation sequence of
  tilting bundles $T=T^{(1)},T^{(2)},\ldots,T^{(r)}$, where
  $\End(T^{(r)})$ is hereditary, accordingly $T^{(r)}$ yields a slice
  in $\Ga$. If the quiver of $\End(T)$ has no relations, then
  $\End(T)$ is hereditary, and we are done. Otherwise we choose an
  ordering of the indecomposable direct summands $T_1,\ldots,T_n$ of
  $T$ such that (a) $\Hom(T_j,T_i)=0$ for $j>i$ and (b) the index $s$
  is minimal among $1,\ldots,n$ such that a relation of $\End(T)$
  starts in $T_s$. The full subquiver $T_1,\ldots,T_s$ will not
  contain any cycle since otherwise $\End(T)$ would be wild,
  contradicting $\chi_\Her>0$. Without changing the component
  $\comp{T}$, by invoking suitable APR-tilts in sources and
  APR-cotilts in sinks of $T_1,\ldots,T_{s-1}$, we may then assume
  that $T_s$ is a source in the quiver of $\End(T)$. Now mutation
  (APR-tilt) in $T_s$ yields a complement $T_s^*$, replacing $T_s$,
  given by an exact sequence $0\ra T_s^* \ra \bigoplus_{q\in J} T_q
  \ra T_s \ra 0$, where the $T_q$ are indecomposable summands of $T$
  not isomorphic to $T_s$.  Passing to ranks we obtain:
\begin{eqnarray*}
\rank(T_s) + \rank(T_s^*) &=& \sum_{s\ra q}\rank(T_q) \\
2\rank(T_s) &=& \sum_{s\ra q}\rank(T_q)-\sum_{s\cdots r}\rank(T_r),
\end{eqnarray*}
where the first formula is obvious and the second one, with summation
over a set of minimal relations starting at $s$ in the second sum,
expresses H\"{u}bner's rank additivity on tilting objects in $\Her$,
see \cite{Huebner2} or \cite{LeRe}. We thus obtain
$$
\rank(T_s^*)-\rank(T_s)=\sum_{s\cdots r}\rank(T_r)>0.
$$
Since for $\chi_\Her>0$ there exists an upper bound on the ranks of
indecomposable bundles, this procedure allows only a finite number of
repetitions, finally yielding a tilting bundle whose endomorphism ring
has no relations. This finishes the proof of the first step.

\emph{Step 2.} Let $T'$ and $T''$ be tilting bundles in $\Her$. By
step~1 there is are mutation sequences of tilting bundles transforming
$T'$ and $T''$ into slices $\bar{T'}$ and $\bar{T''}$, respectively.
It is well known (and easy to see) that any two slices of $\Ga$ are
connected by a mutation sequence in $\Ga$ consisting of
BGP-reflections.  This proves the claim.
\end{proof}

\begin{Proposition} \label{prop:torsion}
Assume that $\chi_\Her\geq0$. Then each tilting object $T$ with a
non-zero torsion part belongs to the connected component $\comp\can$.
\end{Proposition}

\begin{proof}
  Since the torsion part $T_0$ of $T$ is non-zero we may assume by
  Proposition~\ref{prop:reduction} that $T_0$ is an exceptional simple
  object $S$. Invoking Proposition~\ref{prop:Huebner} we may further
  assume that $\Hom(\Oo,S)=k$. Hence $T=T_+\oplus S$ where $T_+$ is a
  bundle.  Mutation of $\can$ at $R=\Oo((p_i-1)\vx_i)$ replaces $R$ by
  $S$ and yields a tilting object $T'=\can'\oplus S$ in $\comp{\can}$
  with the same torsion part as $T$. By~\cite{GeLe2} the right
  perpendicular category $\Her'=S^\perp$, formed in $\Her$, is
  naturally equivalent to a category of coherent sheaves of weight
  type $(p'_1,\ldots,p'_t)$ where $p'_j=p_j$ for $j\neq i$ and
  $p'_{i}=p_i-1$ and such that $\can'$ equals the canonical
  configuration in $\Her'$. Since $\chi_\Her\geq0$ the Euler
  characteristic of $\Her'$ is strictly positive, hence
  Proposition~\ref{prop:domestic} yields a mutation sequence
  $T_+=T^{(1)},T^{(2)},\ldots,T^{(s)}=\can'$ of tilting bundles in
  $\Her'$ connecting $T_+$ and $\can'$. It follows that
  $T=T^{(1)}\oplus S,\,T^{(2)}\oplus S,\ldots,\,T^{(s)}\oplus S=T'$ is
  a mutation sequence connecting $T$ and $T'$. We conclude that $T$
  belongs to $\comp{\can}$.
\end{proof}

Note that the automorphism group $\Aut(\Clu)$, where
$\Clu=\Clu(\Her)$, naturally acts on the tilting graph
$\graph{\Clu}$.
\begin{Proposition} \label{prop:aut_component}
The connected component $\comp{\can}$ is preserved under the action of
$\Aut(\Clu)$.
\end{Proposition}
\begin{proof}
  For $\chi_\Her\neq0$ this is immediate from
  Proposition~\ref{prop:Huebner} and Lemma~\ref{lem:trivial}. For
  $\chi_\Her=0$ we need an additional argument.  Let $\si$ and $\rho$
  be automorphisms of $\Der(\Her)$ such that $\si$ and $\rho$ act on
  slopes by $x\mapsto x+1$ and $x\mapsto x/(1+x)$, respectively; see
  \cite{LeMe2} or \cite{LeMe}. We put $\phi=\rho\si^{1-p}$, where $p$
  is the least common multiple of the weight sequence. Then $\phi$
  sends $\can$ to a tilting object $T'=\phi(\can)$ having maximal
  slope $\infty$, hence a non-zero torsion part. By
  Proposition~\ref{prop:torsion} the tilting object $T'$ belongs to
  the component $\can$.  Hence $\comp{\can}$ is stable under the
  automorphism $\bar\phi$ of $\Clu$ induced by $\phi$ and, invoking
  Proposition~\ref{prop:Huebner}, also under the action of the
  subgroup $\comp{\bar\sigma,\bar\rho}$, which acts transitively on
  the rational circle $\WW$. The claim now follows from
  Proposition~\ref{prop:tub_aut}.
\end{proof}

\begin{Theorem}
Assume $\Her$ has Euler characteristic $\chi_\Her\geq0$. Then the
tilting graph $\graph{\Clu(\Her)}$ is connected.
\end{Theorem}

\begin{proof} Let $T$ be a tilting object in $\Her$. We show that $T$
  belongs to the connected component $\comp{\can}$.

\emph{Case $\chi_\Her>0$.} By Proposition~\ref{prop:reduction} we may
assume that $T$ is a tilting bundle. Then
Proposition~\ref{prop:domestic} proves that $T$ belongs to
$\comp{\can}$, proving the claim in this case.

\emph{Case $\chi_\Her=0$.} Let $T$ be a tilting object in $\Her$. By
an automorphism $\phi$ in $\Der(\Her)$ we achieve that $T'=\phi(T)$
has maximal slope $\infty$, hence by Proposition~\ref{prop:torsion}
belongs to $\comp{\can}$. Since $\comp{\can}$ is closed under the
action of the automorphism group of $\Clu$, it follows that $T$
belongs to $\comp{\can}$.
\end{proof}

\begin{Remark}
For $\chi_\Her>0$ connectedness of $\graph{\Her}$ is known for a long
time, compare \cite{Happel:Unger}
and \cite{5clu}.
For the tubular weight type $(2,2,2,2)$ connectedness of
$\graph{\Her}$ has been shown by Barot and Gei{\ss} using
combinatorial techniques (unpublished). It is conjectured that the
tilting graph is also connected for $\chi_\Her<0$.
\end{Remark}


\end{document}